\documentclass[a4paper,fleqn]{cas-dc}

\usepackage[square,numbers]{natbib}

\usepackage{mathtools}
\usepackage{amsmath,amssymb,amsfonts}
\usepackage{amsthm}
\usepackage{hyperref}
\usepackage{subcaption}

\usepackage[nameinlink,capitalise]{cleveref}
\crefname{equation}{}{} 

\usepackage[normalem]{ulem}
\usepackage{lipsum}

\usepackage{algorithm,algpseudocode}

\usepackage{tikz}
\usetikzlibrary{shapes,arrows,positioning,calc,patterns}

\usepackage{pgfplots}
\pgfplotsset{compat=1.18}
\usepgfplotslibrary{colorbrewer}
\usepgfplotslibrary{fillbetween}
\usepgfplotslibrary{groupplots}

\usepackage{pgfplotstable}
\usepackage{booktabs}
\usepackage{colortbl}

\newtheorem{corollary}{Corollary}
\newtheorem{proposition}{Proposition}
\newtheorem{definition}{Definition}

\newtheorem{theorem}{Theorem}
\newtheorem{remark}{Remark}

\newtheorem{lemma}{Lemma}
\crefname{lemma}{Lemma}{Lemmas} 
\Crefname{lemma}{Lemma}{Lemmas}

\newcommand*\circled[1]{\tikz[baseline=(char.base)]{\node(char)[shape=rounded rectangle,draw,inner sep=0.6pt,minimum height=1.5ex]{#1};}} %
\newcommand\kronF[2]{{#1}^{\circled{\tiny{\ensuremath{#2}}}}} 
\newcommand*\squared[1]{\tikz[baseline=(char.base)]{\node(char)[shape=rectangle,draw,inner sep=0.6pt,minimum height=1.5ex]{#1};}} %
\newcommand\kronFMin[2]{{#1}^{\squared{\tiny{\ensuremath{#2}}}}} 

\renewcommand\Vec[1][]{\textnormal{vec}\ensuremath{\if$#1$ \else \left[#1\right]\fi}}
\newcommand\Vech[1][]{\textnormal{vech}\ensuremath{\if$#1$ \else \left[#1\right]\fi}}
\newcommand\UnVec[1][]{\textnormal{unvec}\ensuremath{\if$#1$ \else \left[#1\right]\fi}}

\newcommand{\real}{\mathbb{R}}

\newcommand{\rd}{\text{\upshape d}} 

\newcommand{\exeq}{\overset{!}{=}}

\newcommand{\bzero}{\ensuremath{\mathbf{0}}} 

\newcommand{\tl}{\ensuremath{\tilde{l}}}

\newcommand{\bA}{\ensuremath{\mathbf{A}}}
\newcommand{\bB}{\ensuremath{\mathbf{B}}}
\newcommand{\bC}{\ensuremath{\mathbf{C}}}
\newcommand{\bD}{\ensuremath{\mathbf{D}}}

\newcommand{\bG}{\ensuremath{\mathbf{G}}}
\newcommand{\bH}{\ensuremath{\mathbf{H}}}
\newcommand{\bI}{\ensuremath{\mathbf{I}}}

\newcommand{\bL}{\ensuremath{\mathbf{L}}}
\newcommand{\bM}{\ensuremath{\mathbf{M}}}
\newcommand{\bN}{\ensuremath{\mathbf{N}}}

\newcommand{\bP}{\ensuremath{\mathbf{P}}}

\newcommand{\bR}{\ensuremath{\mathbf{R}}}

\newcommand{\bT}{\ensuremath{\mathbf{T}}}

\newcommand{\bV}{\ensuremath{\mathbf{V}}}
\newcommand{\bW}{\ensuremath{\mathbf{W}}}

\renewcommand{\bf}{\ensuremath{\mathbf{f}}}
\newcommand{\bg}{\ensuremath{\mathbf{g}}}
\newcommand{\bh}{\ensuremath{\mathbf{h}}}

\newcommand{\bp}{\ensuremath{\mathbf{p}}}

\newcommand{\br}{\ensuremath{\mathbf{r}}}

\newcommand{\bu}{\ensuremath{\mathbf{u}}}
\newcommand{\bv}{\ensuremath{\mathbf{v}}}
\newcommand{\bw}{\ensuremath{\mathbf{w}}}
\newcommand{\bx}{\ensuremath{\mathbf{x}}}
\newcommand{\by}{\ensuremath{\mathbf{y}}}
\newcommand{\bz}{\ensuremath{\mathbf{z}}}

\newcommand{\bPhi}{\mbox{\boldmath $\Phi$}}

\newcommand{\bSigma}{\ensuremath{\boldsymbol{\Sigma}}}

\newcommand{\cA}{\ensuremath{\mathcal{A}}}

\newcommand{\cE}{\ensuremath{\mathcal{E}}}

\newcommand{\cH}{\ensuremath{\mathcal{H}}}

\newcommand{\cT}{\ensuremath{\mathcal{T}}}
\newcommand{\cU}{\ensuremath{\mathcal{U}}}
\newcommand{\cV}{\ensuremath{\mathcal{V}}}
\newcommand{\cW}{\ensuremath{\mathcal{W}}}

\newcommand{\ca}{\ensuremath{\mathcal{a}}}

\newcommand{\cs}{\ensuremath{\mathcal{s}}}

\newcommand{\F}{\ensuremath{\mathbf{F}}}

\newcommand{\nomathindent}{\setlength{\mathindent}{0cm}}

\begin{document}
\let\WriteBookmarks\relax
\def\floatpagepagefraction{1}
\def\textpagefraction{.001}

\shorttitle{Input-normal/output-diagonal transformation}

\shortauthors{N. A. Corbin, A. Sarkar, J. M. A. Scherpen, B. Kramer}

\title[mode = title]{Scalable computation of input-normal/output-diagonal balanced realization for control-affine polynomial systems}

\author[1]{Nicholas A. Corbin}[
    orcid=0000-0002-0621-112X,
    linkedin=nick-corbin]

\cormark[1]

\ead{ncorbin@ucsd.edu}

\ead[url]{https://cnick1.github.io/}

\affiliation[1]{organization={University of California San Diego},
    addressline={9500 Gilman Drive},
    city={La Jolla},
    postcode={92093},
    state={CA},
    country={USA}}

\cortext[1]{Corresponding author}

\author[2]{Arijit Sarkar}[
    orcid=0000-0003-2027-9566,
    linkedin=arijit-sarkar-a2013413a]

\affiliation[2]{organization={Jan C. Willems Center for Systems and Control, Engineering and Technology Institute Groningen, University of Groningen},
    addressline={Nijenborgh 4},
    city={Groningen},
    postcode={9747},
    state={AG},
    country={The Netherlands}}

\ead{a.sarkar@rug.nl}

\author[2]{Jacquelien M. A. Scherpen}[
    orcid=0000-0002-3409-5760,
    linkedin=jacquelien-scherpen-6b0bb06]

\ead{j.m.a.scherpen@rug.nl}

\author[1]{Boris Kramer}[
    orcid=0000-0002-3626-7925,
    linkedin=kramerboris]
\ead{bmkramer@ucsd.edu}

\begin{abstract}
    We present a scalable tensor-based approach to computing input-normal/output-diagonal nonlinear balancing transformations for control-affine systems with polynomial nonlinearities.
    This transformation is necessary to determine the states that can be truncated when forming a reduced-order model.
    Given a polynomial representation for the controllability and observability energy functions,
    we derive the explicit equations to compute a polynomial transformation to induce input-normal/output-diagonal structure in the energy functions in the transformed coordinates.
    The transformation is computed degree-by-degree, similar to previous Taylor-series approaches in the literature.
    However, unlike previous works, we provide a detailed analysis of the transformation equations in Kronecker product form to enable a scalable implementation.
    We derive the explicit algebraic structure for the equations, present rigorous analyses for the solvability and algorithmic complexity of those equations, and provide general purpose open-source software implementations for the proposed algorithms to stimulate broader use of nonlinear balanced truncation model reduction.
    We demonstrate that with our efficient implementation, computing the nonlinear transformation is approximately as expensive as computing the energy functions.
\end{abstract}


\maketitle

\section{Introduction}\label{sec:intro}
With the advent of new technologies, engineering systems are rapidly becoming more and more complex.
Due to their complexity, the associated mathematical models used to design, analyze, and control these engineering systems are often nonlinear and high-dimensional.
This demands the systematic development of reduced-order models (ROMs) which impose less computational burden and can preserve certain properties of the high-dimensional system.

Balanced truncation is a system-theoretic model reduction method developed for control systems because of its innate ability to preserve stability, controllability, and observability.
It was initially introduced for linear time-invariant (LTI) systems by \citet{Mullis1976} and \citet{Moore1981}.
Several different types of balanced realizations and model reduction for LTI systems have been introduced whose detailed exposition can be found in \citet{Gugercin2004}.
Scherpen extended the notion of balancing to stable nonlinear systems in \cite{Scherpen1993,Scherpen1994a}, unstable nonlinear systems in \cite{Scherpen1994},
and $\cH_\infty$ balancing for closed-loop systems was introduced in \cite{Scherpen1996}.
Shortly thereafter, the connection between the nonlinear balancing controllability/observability energy functions and minimality was highlighted in \cite{Gray2000}. In the roughly 10 years that followed, the nonlinear balancing theory was refined to address certain peculiarities that do not appear in linear balancing, particularly relating to the nonlinear coordinate transformations required to balance the nonlinear energy functions \cite{Gray2001,Fujimoto2001a,Fujimoto2003,Fujimoto2005,Fujimoto2010}.

A key step in both linear and nonlinear balancing is the transformation of the system into input-normal/output-diagonal form, in which all of the controllability and observability information coalesces into the squared singular values found along the ``diagonal'' of the transformed observability function.
In nonlinear balancing, however, these squared singular values appear as functions of the state.
The original \emph{smooth singular value functions} \cite{Scherpen1993,Scherpen1994a,Scherpen1994,Scherpen1996} were shown to be non-unique in \cite{Gray2001}, and their connection with the Hankel operator was not clear.
These issues were then resolved in several works by \citet{Fujimoto2001a,Fujimoto2003,Fujimoto2005,Fujimoto2010} with the definition of the \emph{axis singular value functions}.
Unlike the original smooth singular value functions which had full state dependence, each axis singular value function only depends on one component of the state\footnote{
    The differing definitions of singular value functions have led to some ambiguity in the literature regarding the definition of input-normal/output-diagonal realizations for nonlinear systems.
    In this work, only the realization with axis singular value functions will be considered truly input-normal/output-diagonal. }.
With these results, and in particular the nonlinear operator-based formulation presented in \cite{Fujimoto2010}, which generalizes to discrete-time systems, the nonlinear balancing \emph{theory} has been considered mature for some time.

However, the practical application of nonlinear balanced truncation has remained limited, largely because of three computational challenges.
First, nonlinear balancing demands solving a Lyapunov-like and a Hamilton-Jacobi-Bellman (HJB) partial differential equation (PDE) for the controllability and observability functions associated with the nonlinear system.
Second, the theory requires computing a nonlinear coordinate transformation to first put the system in an input-normal/output-diagonal form, followed by another nonlinear scaling of the states to achieve a fully balanced form.
Third, in the transformed coordinates, the ROM must be formed and simulated efficiently.

Regarding the first challenge, tensor-based scalable algorithms are built in \cite{Krener2013,Breiten2018,Almubarak2019,Borggaard2020,Borggaard2021} based on the method of Al'brekht \cite{Albrekht1961,Lukes1969} to solve HJB PDEs by expanding the energy functions as a Taylor series\footnote{While many other approaches exist for solving HJB PDEs, only the Taylor series approach has been shown to be compatible with computing balancing transformations.}.
\citet{Kramer2024} extended this approach to specifically compute the energy functions associated with nonlinear balancing for models with quadratic drift nonlinearity.
Subsequently, \citet{Corbin2023} provided the full generalization to models with arbitrary polynomial nonlinearity in the drift, input, and output maps.
Towards the second challenge, \citet{Fujimoto2006} proposed to approximate the nonlinear coordinate transformation as a Taylor series.
Similar to how the polynomial coefficients of the energy function can be computed degree-by-degree using Al'brekht's method, they show that the coordinate transformation coefficients can also be computed degree-by-degree.
However, the method is merely outlined abstractly; conditions for the existence of the solutions to the coefficients, details regarding the practical computation of the transformation coefficients, and a scalable implementation as needed for model reduction are not provided.
Independently, \citet{Krener2008} also presented a Taylor series-based, degree-by-degree approach to computing polynomial input-normal/output-diagonal transformations.
That work includes a more precise analysis regarding the existence of solutions to the transformation coefficients.
It also explains more about uniqueness considerations for the singular value functions computed using this approach.
While the core idea is the same as in \cite{Fujimoto2006}, Krener
more precisely describes how individual components of the polynomial transformation work to eliminate particular monomial components and put the system in input-normal/output-diagonal form.
In practice,
building the transformation in this piece-by-piece manner is cumbersome and not scalable.
As a result, nonlinear balancing to date has yet to be demonstrated on systems larger than $n=6$ \cite{Krener2008}.
Consequently, the third challenge of efficiently forming and simulating ROMs has also never been addressed for systems larger than $n=6$; we plan to address this in future work.

This paper contributes a \textit{scalable} method for the second challenge of finding a polynomial transformation based on the Kronecker product, thus improving on the approaches of \cite{Fujimoto2006,Krener2008}.
\cref{sec:main-result} presents the explicit algebraic equations for the input-normal/output-diagonal transformation coefficients.
\cref{sec:secondary-results} provides a detailed proof for the existence of a transformation satisfying those equations.
This section also includes an analysis of the computational complexity of the proposed approach based on an efficient implementation leveraging tensor structure in the equations.
In \cref{sec:results}, we demonstrate our proposed approach first with a two-dimensional academic example from the literature, and then with a coupled mass-spring-damper system to demonstrate the scalability of the approach.
Finally, we draw conclusions in \cref{sec:conclusion}.

\section{Preliminaries}\label{sec:preliminaries}
A summary of the most relevant results from nonlinear balancing theory is presented in \cref{sec:nl-balancing-theory}.
Then in \cref{sec:notation},
we define several concepts relating to Kronecker product algebra that we leverage for efficient computation.

\subsection{Nonlinear balancing theory}\label{sec:nl-balancing-theory}
Consider the nonlinear control-affine dynamical system
\begin{equation}\label{eq:FOM-NL}
    \dot{\bx}(t)  = \bf(\bx(t))  + \bg(\bx(t)) \bu(t), \qquad
    \by(t)        = \bh(\bx(t)),
\end{equation}
where $\bx(t) \in \real^n$ is the state, $\bf \colon \real^n \to \real^n$ is the nonlinear drift vector field, $\bg \colon \real^n \to \real^{n \times m}$ is the matrix-valued function whose columns are nonlinear input vector fields corresponding to each control input, $\bh \colon \real^n \to \real^{p}$ is the nonlinear output map, $\bu(t) \in \real^m$ is a vector of input signals, and $\by(t) \in \real^p$ is a vector of output signals.

Following \citet{Scherpen1993}, the controllability and observability energy functions for an asymptotically stable nonlinear system \cref{eq:FOM-NL} are
\footnote{We focus on the ``open-loop'' nonlinear balancing theory to simplify notation.
    However, all of the results in this article extend to the more general closed-loop and $\cH_\infty$ balancing frameworks \cite{Scherpen1994,Scherpen1996}.}
\begin{align*}
    \cE_c(\bx_0) & \coloneqq
    \!\!\!\! \min_{\substack{\bu \in L_{2}
            (-\infty, 0)
    \\ \bx(-\infty) = \bzero,\,\,   \bx(0) = \bx_0}}
    \!\!
    \frac{1}{2} \int\displaylimits_{-\infty}^{0} \Vert \bu(t) \Vert^2 \rd t,
    \hspace{-2cm}
    \\
    \cE_o(\bx_0) & \coloneqq \frac{1}{2} \int\displaylimits_{0}^{\infty} \Vert \by(t) \Vert^2  \rd t,
    \begin{split}
         & \qquad\bx(0) = \bx_0, \\&\bu(t) \equiv \bzero, \,\,0 \leq t < \infty.
    \end{split}
\end{align*}
Assuming that these functions exist, i.e. are finite, and are smooth,
$\cE_c(\bx)$
is the unique solution to the
Hamilton-Jacobi equation
\begin{equation}\label{eq:HJ-Equation}
    \begin{split}
         & 0 =  \frac{\partial \cE_c(\bx)}{\partial \bx} \bf(\bx) + \frac{1}{2}  \frac{\partial \cE_c(\bx)}{\partial \bx} \bg(\bx) \bg(\bx)^\top \frac{\partial^\top \cE_c(\bx)}{\partial \bx}, \\
         & \cE_c(\bzero) = 0,
    \end{split}
\end{equation}
under the assumption that $0$ is an asymptotically stable equilibrium of $-\left(\bf(\bx)+\bg(\bx)\bg(\bx)^\top\frac{\partial^\top\cE_c(\bx)}{\partial \bx}\right)$, and $\cE_o(\bx)$ is the unique solution to the nonlinear Lyapunov-like equation
\begin{equation}\label{eq:Nonlin-Lyap}
    \begin{split}
         & 0 =  \frac{\partial \cE_o(\bx)}{\partial \bx} \bf(\bx) + \frac{1}{2}\bh(\bx)^\top \bh(\bx), \\
         & \cE_o(\bzero) = 0.
    \end{split}
\end{equation}
The next theorem states the existence of a coordinate transformation in which the energy functions take on a special \emph{input-normal/output-diagonal} form.
\begin{theorem}\cite[Thm. 8]{Fujimoto2010} \label{thm:inputnormaloutputdiagonal}
    Suppose the Jacobian linearization of the nonlinear system \cref{eq:FOM-NL} is minimal, asymptotically stable, and has distinct Hankel singular values.
    Then there is a neighborhood $\cW$ of the origin and a smooth coordinate transformation $\bx = \Phi(\bz)$ on $\cW$ with $\bz = [z_1, z_2, \dots, z_n]$ such that the controllability and observability energy functions have \emph{input-normal/output-diagonal} form:
    \small
    \begin{align}
        \hspace*{-.5cm} \cE_c(\Phi(\bz)) & = \frac{1}{2} \bz^\top \bz             &  & = \frac{1}{2} \sum_{i=1}^n  z_i^2,
        \\
        \hspace*{-.5cm} \cE_o(\Phi(\bz)) & = \frac{1}{2} \bz^\top \begin{bmatrix}
                                                                      \sigma^2_1(z_1)             \\
                                                                       & \ddots                   \\
                                                                       &        & \sigma^2_n(z_n) \\
                                                                  \end{bmatrix}\bz &  & = \frac{1}{2} \sum_{i=1}^n  z_i^2 \sigma_i^2(z_i). \label{eq:in-od-definition}
    \end{align}
    \normalsize
\end{theorem}
\begin{remark}
    The functions $\sigma_i(z_i)$ are known as the \emph{axis singular value functions} \cite{Fujimoto2010}; critically, they only depend on the $i$th component of the transformed state $z_i$.
    This is in distinction with the original \emph{smooth singular value functions} $\tau_i (\bz)$ defined in \cite{Scherpen1993} which could have full state dependence,
    resulting in non-uniqueness of the transformations and the resultant nonlinear ROMs.
\end{remark}

In general, solving the PDEs \cref{eq:HJ-Equation,eq:Nonlin-Lyap} analytically for the energy functions is not feasible.
In this work, we leverage the approximation technique of Al'brekht.
For the special case when the functions $\bf(\bx)$, $\bg(\bx)$, and $\bh(\bx)$ of the dynamics \cref{eq:FOM-NL} are analytic, the energy functions are also analytic \cite{Albrekht1961,Lukes1969}, in which case they can be
expressed in a convergent power series which we truncate to $d$th order as
\begin{align}
     & \cE_c(\bx)
    \approx \frac{1}{2} \sum_{i=2}^d \bv_i^\top \kronF{\bx}{i},
     &            &
     & \cE_o(\bx)
    \approx \frac{1}{2} \sum_{i=2}^d \bw_i^\top \kronF{\bx}{i}. \label{eq:vi-coeffs}
\end{align}
Algorithm 1 in \cite{Corbin2023} provides an efficient method for computing the coefficients $\bv_i$ and $\bw_i$, and an open-source \textsc{Matlab} implementation is available in the \texttt{cnick1/NLbalancing} repository \cite{NLBalancing2023}.
Henceforth, we assume that the energy functions are available in the form \cref{eq:vi-coeffs}.

\subsection{Kronecker product definitions and notation}\label{sec:notation}
The Kronecker product of two matrices $\bA \in \real^{p \times q}$ and $\bB \in \real^{s \times t}$ is the $ps \times qt$ block matrix
\begin{align*}
    \bA \otimes \bB \coloneqq \begin{bmatrix} a_{11}\bB & \cdots & a_{1q}\bB \\
                \vdots    & \ddots & \vdots    \\
                a_{p1}\bB & \cdots & a_{pq}\bB
                              \end{bmatrix},
\end{align*}
where $a_{ij}$ denotes the $(i,j)$th entry of $\bA$.
We write $k$-times repeated Kronecker products as
$\kronF{\bx}{k} \coloneqq \bx \otimes \dots \otimes \bx \in \real^{n^k}$.
For a variable $\bx \in \real^n$, the quantity $\kronF{\bx}{k}$ is a vector whose entries are all monomials of degree $k$, and
for a coefficient vector $\bw_k \in \real^{n^k}$, the quantity $\bw_k^\top \kronF{\bx}{k}$ is a scalar homogeneous polynomial of degree $k$.

\begin{definition}[Diagonal monomial entries]\label{def:diagonal}
    Consider the set of univariate monomial entries of $\kronF{\bx}{k}$, i.e. the entries $\{x_1^k, x_2^k, \dots, x_n^k\}$.
    We refer to these as \underline{\emph{diagonal}} entries.
\end{definition}
\begin{definition}[Off-diagonal monomial entries]\label{def:off-diagonal}
    We refer to the remaining multivariate monomial entries of $\kronF{\bx}{k}$ as \underline{\emph{off-diagonal}} entries.
    These are entries of the form
    $x_{i_1}^{e_1} x_{i_2}^{e_2} \dots x_{i_n}^{e_n}$, where the indices $\left\{i_1, \dots, i_n\right\} \in \left\{1,\dots,n\right\}$ contain at least two distinct values and the exponents $\left\{e_1, \dots, e_n\right\} \in \left\{0,\dots,k\right\}$ sum to $k$.
\end{definition}

We introduce \cref{fig:diag-tensors} to illustrate the motivation for distinguishing between diagonal and off-diagonal entries.
If the vector $\kronF{\bx}{k} \in \real^{n^k}$ is reshaped as an $n\times n \times \dots \times n$ tensor of dimension $k$,
the entries along the diagonal of the tensor are the univariate monomials, hence why we refer to them as diagonal entries.
\begin{figure}[htb]
    \centering
    \begin{subfigure}[h]{0.45\columnwidth}
        \centering
        \includegraphics[width=\textwidth,page=2]{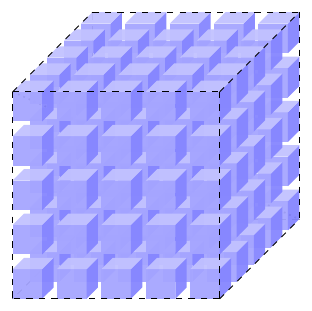}
        \caption{Unstructured tensor.}
        \label{sfig:label}
    \end{subfigure}
    \hfill
    \begin{subfigure}[h]{0.45\columnwidth}
        \centering
        \includegraphics[width=\textwidth,page=4]{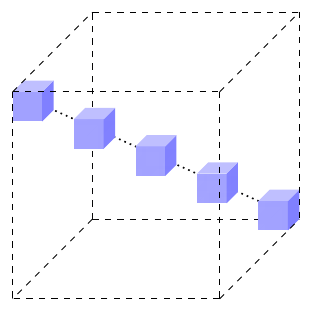}
        \caption{Diagonal tensor.}
    \end{subfigure}
    \caption{Visualization of non-diagonal and diagonal tensors in three dimensions.}
    \label{fig:diag-tensors}
\end{figure}

\noindent Similarly, for a scalar homogeneous polynomial $\bw_k^\top \kronF{\bx}{k}$, the coefficient vector $\bw_k \in \real^{n^k}$ can be reshaped in the same way as an $n\times n \times \dots \times n$ tensor of dimension $k$.
If the only nonzero entries in $\bw_k$ are the diagonal entries, then coefficient $\bw_k$ can be said to have diagonal structure (analogous to a diagonal matrix).
Consequently, the scalar homogeneous polynomial $\bw_k^\top \kronF{\bx}{k}$ can be thought of as ``diagonal'', leading to the following definition formalizing this concept for general polynomials.
\begin{definition}[Diagonal polynomial]\label{def:diagonal-polynomial}
    Let $p : \real^n \to \real$ be a scalar polynomial of degree $d$, which we write as $p(\bx) = \sum_{i=1}^d \bw_i^\top \kronF{\bx}{i}$.
    The polynomial $p(\bx)$ is said to be \underline{\emph{diagonal}} if all of its coefficients $\bw_i \in \real^{n^i}$ have diagonal structure.
    In other words, a polynomial is said to be diagonal if it is a sum of univariate monomials.
\end{definition}

Definitions \labelcref{def:diagonal}, \labelcref{def:off-diagonal}, and \labelcref{def:diagonal-polynomial} play a central role in the present work: to ensure input-normal/output-diagonal structure, we will impose diagonal structure on the polynomial observability energy function according to \cref{def:diagonal-polynomial}.

Notice that due to the definition of the Kronecker product, off-diagonal monomials appear multiple times in the vector $\kronF{\bx}{k}$.
For example, $\kronF{\bx}{2}$ contains both $x_1 x_2$ and $x_2 x_1$.
The next definition is introduced for scenarios in which we wish to eliminate these duplicate entries.
\begin{definition}[Minimal monomial representation]\label{def:min-monomial}
    Given $\kronF{\bx}{k} \in \real^{n^k}$, consider a vector of reduced dimension ${n+k-1}\choose{k}$ containing only the unique monomial entries.
    We call this the \emph{minimal monomial representation} and write it as $\kronFMin{\bx}{k}$.
\end{definition}

It is sometimes convenient to write the homogeneous polynomial $\bw_k^\top \kronF{\bx}{k}$ using the minimal monomial representation $\kronFMin{\bx}{k}$ instead;
to that end, we introduce the matrix $\bN_k$, a binary matrix with ${n+k-1}\choose{k}$ rows and $n^k$ columns satisfying
\begin{align}
    (\bw_k)^\top \kronF{\bx}{k} = (\bN_k\bw_k)^\top \kronFMin{\bx}{k}. \label{eq:Nhat}
\end{align}
For convenience, we order $\bN_k$ and $\kronFMin{\bx}{k}$ such that the first $n$ rows correspond to the diagonal monomials.
Subsequently, the matrix $\hat\bN_k$ of the remaining rows of $\bN_k$ corresponding to the off-diagonal entries can be defined as
\begin{align*}
    \hat\bN_k \coloneqq \bN_k(\texttt{n+1:end}).
\end{align*}

Two common operations that aid with tensor algebra are the $\Vec(\cdot)$ and $\rm{reshape}(\cdot)$ operations.
The $\Vec(\cdot)$ operation stacks the columns of a matrix into one tall column vector.
As defined in \cite{Golub2013}, for $p_1 q_1 = p_2 q_2$, the $\rm{reshape}(\cdot)$ operation reshapes a matrix $\bA \in \real^{p_1 \times q_1}$ into a new matrix $\tilde\bA = \rm{reshape}(\bA,p_2,q_2) \in \real^{p_2 \times q_2}$ such that $\Vec(\bA) = \Vec(\tilde\bA)$.

\begin{remark}
    The notations
    $\bN_k$
    and
    $\kronFMin{\bx}{k}$
    are reminiscent of
    the ``duplication matrix''
    and
    the $\Vech(\cdot)$ operation
    \cite{Henderson1979}.
    The ``duplication matrix'' $\bD \in \real^{n(n+1)/2 \times n^2}$, defined to satisfy $\bD \Vech(\bM) = \Vec(\bM)$ for a symmetric matrix $\bM\in \real^{n \times n}$, is a special case of $\bN_k^\top$ for $k=2$.
    If the $\Vech(\cdot)$ operation is generalized from symmetric matrices to symmetric tensors and $\kronF{\bx}{k}$ is thought of as an $n\times n \times \dots \times n$ tensor of dimension $k$ rather than an $n^k \times 1$ vector, $\kronFMin{\bx}{k}$ can be regarded as $\Vech(\kronF{\bx}{k})$.
\end{remark}
For a set of $\bT_i \in \real^{n \times n^i}$, we adopt the notation
\begin{equation} \label{eq:tensorproducts}
    \cT_{m,l} \coloneqq \sum_{\sum{i_j}=l} \bT_{i_1}\otimes \dots \otimes \bT_{i_m} \in \real^{n^m\times n^l}
\end{equation}
to compactly write the sum of all unique tensor products with $m$ terms and $n^l$ columns, as introduced in \cite{Kramer2023}.
\begin{lemma}\label{prop:poly-transformation}
    Let $\bp : \real^n \to \real^n$ be a vector-valued polynomial function of degree $d$, which we write as
    \begin{align*}
        \bp(\bx) & =  \sum_{i=1}^d \bP_i \kronF{\bx}{i}, \quad \text{where } \quad \bP_i \in \real^{n \times n^i}.
    \end{align*}
    Additionally, consider a degree $k$ polynomial coordinate transformation $\bPhi: \real^n \to \real^n$ given by
        {\begin{align*}
                \bx  = \Phi(\bz) & = \sum_{i=1}^k \bT_{i} \kronF{\bz}{i},
                \quad \text{where } \quad \bT_i \in \real^{n \times n^i}.
            \end{align*}}
    In the new $\bz$ coordinates, $\bp(\Phi(\bz))$ can be written to degree $d$ as
    \begin{align*}
        \bp(\Phi(\bz)) & =  \sum_{i=1}^d \tilde\bP_i \kronF{\bz}{i}
        , \quad \text{where } \quad \tilde\bP_i  =  \sum_{j=1}^i  \bP_j \cT_{j,i}.
    \end{align*}
\end{lemma}
\begin{proof}
    Substituting $\bx  = \Phi(\bz)$ into $\bp(\bx)$ yields
    \begin{align*}
         & \bp(\Phi(\bz))  =  \sum_{i=1}^d \bP_i \kronF{(\Phi(\bz))}{i}                                             \\
         & =  \bP_1 (\Phi(\bz)) +  \bP_2 \kronF{(\Phi(\bz))}{2}
        + \dots                                                                                                     \\
         & =  \bP_1 (\bT_1 \bz + \bT_2 \kronF{\bz}{2} + \dots + \bT_{k} \kronF{\bz}{k})                             \\
         & \quad +  \bP_2 \left((\bT_1 \bz + \bT_2 \kronF{\bz}{2} + \dots + \bT_{k} \kronF{\bz}{k}) \otimes \right. \\
         & \hspace{1.5cm}\left.(\bT_1 \bz + \bT_2 \kronF{\bz}{2} + \dots + \bT_{k} \kronF{\bz}{k})\right)
        + \dots
    \end{align*}
    Collecting the terms of degree $i$, one can see that the coefficient of $\kronF{\bz}{i}$ for $1 \leq i \leq d$ is $\tilde\bP_i  =  \sum_{j=1}^i  \bP_j \cT_{j,i}$.
\end{proof}
\begin{remark}
    A degree $d$ function in the original $\bx$ coordinates only remains exactly of degree $d$ in the transformed $\bz$ coordinates if the transformation is linear, i.e. of degree $k=1$.
    If the transformation is of degree $k \geq 2$, the degree $d$ function in the original $\bx$ coordinates becomes degree $k\cdot d$ in the transformed $\bz$ coordinates.
    In practice, the degree $d$ of the original function is often chosen as a reflection of the available computational resources, since the computational burden increases with $d$.
    Therefore in this work, we truncate the transformed function also to degree $d$ as an approximation.
\end{remark}

\section{Main result: input-normal/output-diagonal transformation}\label{sec:main-result}
Given the two energy functions \cref{eq:vi-coeffs} in the original $\bx$ coordinates,
we seek to compute a polynomial transformation
\begin{align}\label{eq:transformation}
    \bx  = \Phi(\bz) & = \bT_1 \bz + \bT_2 \kronF{\bz}{2} + \dots + \bT_{d-1} \kronF{\bz}{d-1}
\end{align}
with $\bT_i \in \real^{n \times n^i}$
such that the transformed energy functions are input-normal and output-diagonal
to degree $d$, as in \cref{eq:in-od-definition}.
The next theorem provides this main result.
\begin{theorem}\label{thm:input-normal-output-diagonal-transformation}
    Let $\bv_i,\bw_i \in \real^{n^i}$ for $i=2,3,\dots,d$ be the polynomial coefficients for the controllability and observability energy functions $\cE_c(\bx)$ and $\cE_o(\bx)$ which solve the Hamilton-Jacobi equation \cref{eq:HJ-Equation} and the nonlinear Lyapunov equation \cref{eq:Nonlin-Lyap} for the control-affine system \cref{eq:FOM-NL}.
    Let $\bV_2 = \bR \bR^\top$ and $\bW_2 = \bL \bL^\top$, where $\bR$ and $\bL$ are square-root factors of $\bV_2$ and $\bW_2$.
    Compute
    $\bL^\top \bR^{-\top} =\cU \bSigma \cV^\top$
    using the singular value decomposition (SVD).
    Define the linear transformation coefficient $\bT_1 \coloneqq \bR^{-\top} \cV$, and let
    $\bT_{k-1}$ for $k \geq 3$ satisfy both the \emph{input-normal equations}
    \begin{equation}\label{eq:input-normal-0}
        \begin{split}
             & 2  \bN_k \left(\bI_{n^{k-1}}  \otimes \bT_1^\top\bV_2\right)\Vec \left( \bT_{k-1} \right) \\
             & = - \bN_k \left(\sum_{\substack{i,j\geq 2                                                 \\ i+j=k}}  \Vec \left( \bT_{j}^\top \bV_2 \bT_{i} \right) +  \sum_{i=3}^k  \cT_{i,k}^\top \bv_i\right)  ,
        \end{split}
    \end{equation}
    and the \emph{output-diagonal equations}
    \begin{equation}\label{eq:output-diagonal-0}
        \begin{split}
             & 2  \hat\bN_k \left(\bI_{n^{k-1}}  \otimes \bT_1^\top\bW_2\right)\Vec \left( \bT_{k-1} \right) \\
             & = - \hat\bN_k \left(\sum_{\substack{i,j\geq 2                                                 \\ i+j=k}}  \Vec \left( \bT_{j}^\top \bW_2 \bT_{i} \right) +  \sum_{i=3}^k  \cT_{i,k}^\top \bw_i\right)  .
        \end{split}
    \end{equation}
    Then input-normal/output-diagonal polynomial transformation to degree $d$ is given by \cref{eq:transformation}.
\end{theorem}
\begin{remark}
    The factors $\bR^{-1}$ and $\bL$ in \cref{thm:input-normal-output-diagonal-transformation} can be computed directly using the Hammarling algorithm \cite{Hammarling1982}, i.e. \texttt{lyapchol()} in \textsc{Matlab}.
\end{remark}
\begin{proof}
    We solve for the transformation coefficients by inserting the transformation \cref{eq:transformation} into the controllability and observability energy functions and imposing the input-normal and output-diagonal structures from \cref{thm:inputnormaloutputdiagonal}.
    The transformed energy functions are
    \begin{align*}
         & \hspace{-.5cm} \cE_c(\Phi(\bz))  = \frac{1}{2}\sum_{k=2}^d \tilde\bv_k^\top \kronF{\bz}{k} ,\quad
         & \cE_o(\Phi(\bz))  = \frac{1}{2}\sum_{k=2}^d \tilde\bw_k^\top \kronF{\bz}{k},
    \end{align*}
    where $\tilde\bv_k = \sum_{i=2}^k  \cT_{i,k}^\top \bv_i$ and $\tilde\bw_k = \sum_{i=2}^k  \cT_{i,k}^\top \bw_i$ according to \cref{prop:poly-transformation}.
    The input-normal/output-diagonal structure defined in \cref{eq:in-od-definition} requires that these satisfy the conditions
    \begin{align}
        \sum_{k=2}^d \tilde\bv_k^\top \kronF{\bz}{k} & \exeq \bz^\top \bz,
        \label{eq:input-normal-1}                                                                  \\
        \sum_{k=2}^d \tilde\bw_k^\top \kronF{\bz}{k} & \exeq \sum_{i=1}^n   z_i^2 \sigma^2_i(z_i).
        \label{eq:output-diagonal-1}
    \end{align}
    Assuming that the squared singular value functions can be expanded as polynomials with $\sigma^2_i(z_i)^{\{i\}}$ denoting the degree~$i$ coefficient
    \begin{align*}
        \sigma^2_i(z_i) = \sigma^2_i(0) + \sigma^2_i(z_i)^{\{1\}} z_i  + \sigma^2_i(z_i)^{\{2\}} z_i^2   + \dots,
    \end{align*}
    we proceed to solve for the transformation coefficients $\bT_i$ such that they satisfy \emph{both}\footnote{In \cite{Kramer2023}, only the input-normal equations \cref{eq:input-normal-1} are satisfied. Thus the realization computed therein corresponds to the original \emph{smooth singular value function} formulation rather than the more recent \emph{axis singular value function} formulation. }
    \cref{eq:input-normal-1,eq:output-diagonal-1}.

    Collecting the degree $2$ terms from \cref{eq:input-normal-1,eq:output-diagonal-1},
    the transformed quadratic coefficients $\tilde\bv_2=\bv_2^\top (\bT_1 \otimes \bT_1 )$ and $\tilde\bw_2=\bw_2^\top (\bT_1 \otimes \bT_1 )$
    give
    the following set of conditions on the transformation coefficient $\bT_1$:
    \begin{align*}
        \bv_2^\top (\bT_1 \otimes \bT_1 )\kronF{\bz}{2} & = \bz^\top\bT_1^\top \bV_2 \bT_1 \bz &  & \exeq \bz^\top \bz,           \\
        \bw_2^\top (\bT_1 \otimes \bT_1 )\kronF{\bz}{2} & = \bz^\top\bT_1^\top \bW_2 \bT_1 \bz &  & \exeq \bz^\top \bSigma^2 \bz.
    \end{align*}
    Thus the solution for $\bT_1$ corresponds exactly to the input-normal/output-diagonal transformation used for linear systems.
    Here, we adopt the square-root balancing approach of \citet{Tombs1987}:
    let $\bV_2 = \bR \bR^\top$ and $\bW_2 = \bL \bL^\top$, where $\bR$ and $\bL$ are square-root factors of $\bV_2$ and $\bW_2$.
    Using the SVD to compute $\bL^\top \bR^{-\top} =\cU \bSigma \cV^\top$, it is well known that an input-normal/output-diagonal transformation for the Jacobian linearization of the system is given by
    \begin{align*}
        \bT_1 & = \bR^{-\top} \cV & \text{and} &  & \bT_1^{-1} & = \bSigma^{-1} \cU^\top \bL^\top= \cV^\top \bR^{\top}.
    \end{align*}

    For the remaining coefficients, the collection of degree $k$ terms from \cref{eq:input-normal-1,eq:output-diagonal-1} gives the input-normal/output-diagonal equations for $k \geq 3$ as
    \begin{align}
        0                                             & = \left( \sum_{i=2}^k \cT_{i,k}^\top \bv_i  \right)^\top \kronF{\bz}{k}, \label{eq:input-normal-4}   \\
        \sum_{i=1}^n z_i^2 \sigma^2_i(z_i)^{\{ k-2\}} & = \left( \sum_{i=2}^k \cT_{i,k}^\top \bw_i \right)^\top \kronF{\bz}{k}. \label{eq:output-diagonal-4}
    \end{align}
    Note that the unknown coefficient $\bT_{k-1}$ is contained in $\cT_{i,k}$ in each equation; for now, we will leave it embedded while we manipulate terms to save space.

    It is crucial here to recognize that the Kronecker product representation of the polynomials on the right-hand-sides of \cref{eq:input-normal-4,eq:output-diagonal-4} are not unique due to repetition of off-diagonal entries in $\kronF{\bz}{k}$.
    Using the minimal monomial representation $\kronFMin{\bz}{k}$ and the matrix $\bN_k$ introduced in \cref{sec:notation} to combine these repeated entries,
    the input-normal/output-diagonal equations \cref{eq:input-normal-4,eq:output-diagonal-4} can be rewritten
    \begin{align}
        0 & = \left(  \bN_k\sum_{i=2}^k \cT_{i,k}^\top \bv_i  \right)^\top \kronFMin{\bz}{k}, \label{eq:input-normal-5}   \\
        \sum_{i=1}^n z_i^2 \sigma^2_i(z_i)^{\{ k-2\}}
          & = \left( \bN_k\sum_{i=2}^k \cT_{i,k}^\top \bw_i \right)^\top \kronFMin{\bz}{k}. \label{eq:output-diagonal-5a}
    \end{align}
    If we only consider the off-diagonal entries by using $\hat\bN_k$ instead of $\bN_k$, the output-diagonal equation \cref{eq:output-diagonal-5a} reduces to
    \begin{align}
        0 & = \left( \hat\bN_k\sum_{i=2}^k \cT_{i,k}^\top \bw_i \right)^\top \kronFMin{\bz}{k}. \label{eq:output-diagonal-5b}
    \end{align}

    Since the minimal monomial representation and the matrix $\bN_k$ defined in \cref{eq:Nhat} eliminates the repeated entries, a necessary \emph{and} sufficient condition\footnote{In \cite{Kramer2023}, the coefficient vector of \cref{eq:input-normal-4} (without $\bN_k$) is taken to be zero. This leads to a fully determined linear system (which therefore has a unique solution) for each transformation coefficient. But due to the repeated entries in $\kronF{\bz}{k}$, the unique solution to those linear systems is not the only transformation satisfying \cref{eq:input-normal-4}; furthermore, that solution does not generally satisfy \cref{eq:output-diagonal-4}.} for satisfying \cref{eq:input-normal-5,eq:output-diagonal-5b}
    is for the coefficient vectors to be zero, leading to the equations
    \begin{align}
        \bzero & =  \bN_k\sum_{i=2}^k \cT_{i,k}^\top \bv_i    ,
               &
        \bzero & = \hat\bN_k\sum_{i=2}^k \cT_{i,k}^\top \bw_i . \label{eq:input-normal-output-diagonal-5}
    \end{align}

    To solve these equations for the unknown coefficient $\bT_{k-1}$, we expand the first term in the sums on the
    right-hand sides of both equations using Kronecker product identities:
    \begin{align*}
        \cT_{2,k}^\top \bv_2 & = 2\Vec \left( \bT_{1}^\top \bV_2 \bT_{k-1} \right) + \sum_{\mathclap{\substack{i,j\geq 2 \\ i+j=k}}}  \Vec \left( \bT_{j}^\top \bV_2 \bT_{i} \right).
    \end{align*}
    Together with the analogous expansion for $\bw_2$, the input-normal/output-diagonal equations \cref{eq:input-normal-output-diagonal-5} are rearranged to give
    \begin{align}
         & \begin{split}
                & 2\bN_k\Vec \left( \bT_{1}^\top \bV_2 \bT_{k-1} \right) \\
                & =- \bN_k
               \left(\sum_{\substack{i,j\geq 2                           \\ i+j=k}}  \Vec \left( \bT_{j}^\top \bV_2 \bT_{i} \right) + \sum_{i=3}^k \cT_{i,k}^\top \bv_i\right) \\
           \end{split}, \label{eq:input-normal-6} \\
         & \begin{split}
                & 2\hat\bN_k\Vec \left( \bT_{1}^\top \bW_2 \bT_{k-1} \right) \\
                & =- \hat\bN_k
               \left(\sum_{\substack{i,j\geq 2                               \\ i+j=k}}  \Vec \left( \bT_{j}^\top \bW_2 \bT_{i} \right) + \sum_{i=3}^k \cT_{i,k}^\top \bw_i\right)
           \end{split}.\label{eq:output-diagonal-6}
    \end{align}
    Applying the identity
    $\Vec \left( \bA \bD \right) = \left(\bI  \otimes \bA\right)\Vec \left( \bD \right) $
    to the left-hand-sides isolates the unknown vector $\Vec[\bT_{k-1}]$ and puts the equations in the standard form for a linear system as in \cref{eq:input-normal-0,eq:output-diagonal-0}, completing the proof.

\end{proof}

\begin{remark}\label{rem:uniqueness}
    The unknown transformation coefficient $\bT_{k-1}$ contains $n^k$ unknowns.
    There are ${n+k-1}\choose{k}$ input-normal equations in \cref{eq:input-normal-0} and ${{n+k-1}\choose{k}} - n$ output-diagonal equations in \cref{eq:output-diagonal-0}; in total, these equations are less than or equal to the $n^k$ unknowns.
    Due to the redundant terms in the Kronecker product form, we can impose an additional
    $n^k - n{{n+k-2}\choose{k-1}}$ constraints to eliminate the redundant terms.
    However, for $n \geq 3$, the number of equations is still less than $n^k$, meaning that in general the transformation coefficients are not uniquely determined.
    This is consistent with \cite[Theorem 7]{Fujimoto2005}: for $n \geq 3$, the state-dependent eigenvalue problem described therein does not have a unique solution, hence the coordinate transformations to achieve an input-normal output-diagonal balanced realization are not unique.
    This is also consistent with the results of \citet{Fujimoto2006} and \citet{Krener2008}.
    While in theory any of the solutions will yield an admissible transformation, we discuss the particular solution that our algorithm selects in \cref{sec:flops}.
\end{remark}

The input-normal/output-diagonal equations \cref{eq:input-normal-0,eq:output-diagonal-0}
represent a \textbf{dense} underdetermined linear system for each unknown coefficient $\bT_{k-1}$ for $k \geq 3$.
The next corollary suggests a subtle implementation step to render these equations sparse to significantly reduce their computational burden.
\begin{corollary}\label{prop:sparse computation}
    After computing the linear transformation coefficient $\bT_1$, the linear transformation $\bT_1$ is applied to the energy function coefficients $\bv_i$ and $\bw_i$.
    In the transformed coordinates, the quadratic components of the energy functions are input-normal ($\tilde\bV_2 = \bI$) and output-diagonal ($\tilde\bW_2 = \bSigma^2$).
    The remaining energy function coefficients $\tilde\bv_i$ and $\tilde\bw_i$ for $3 \leq i \leq d$ can be computed according to \cref{prop:poly-transformation}.
    In these transformed coordinates, the input-normal and output-diagonal equations for the remaining transformation coefficients $\hat\bT_{k-1}$ for $k \geq 3$ are
    \begin{equation}\label{eq:input-normal-9}
        \begin{split}
             & 2  \bN_k \Vec \left( \hat\bT_{k-1} \right) \\
             & = - \bN_k \left(\sum_{\substack{i,j\geq 2  \\ i+j=k}}  \Vec \left( \hat\bT_{j}^\top \hat\bT_{i} \right) +  \sum_{i=3}^k  \hat\cT_{i,k}^\top \tilde\bv_i\right),
        \end{split}
    \end{equation}
    and
    \begin{equation}\label{eq:output-diagonal-9}
        \begin{split}
             & 2  \hat\bN_k \left(\bI_{n^{k-1}}  \otimes \bSigma^2\right)\Vec \left( \hat\bT_{k-1} \right) \\
             & = - \hat\bN_k \left(\sum_{\substack{i,j\geq 2                                               \\ i+j=k}}  \Vec \left( \hat\bT_{j}^\top \bSigma^2 \hat\bT_{i} \right) +  \sum_{i=3}^k  \hat\cT_{i,k}^\top \tilde\bw_i\right)  .
        \end{split}
    \end{equation}
    Upon computing the secondary transformation given by the coefficients $\hat\bT_i$, the coefficients of the full input-normal/output-diagonal transformation of degree $d$ are
    $\bT_i = \bT_1 \hat\bT_i$.
\end{corollary}
It is straightforward to verify that both approaches are equivalent.
However, the sparsity enabled by the two-step approach greatly accelerates the computations, so in practice \cref{prop:sparse computation} should always be used over \cref{thm:input-normal-output-diagonal-transformation}.

\section{Analysis of the solvability and scalability of the proposed approach}\label{sec:secondary-results}
In this section, we discuss details regarding the solution of the input-normal/output-diagonal equations presented in \cref{thm:input-normal-output-diagonal-transformation}.
We begin in \cref{sec:existence} by proving the existence of solutions to the transformation equations of \cref{thm:input-normal-output-diagonal-transformation}.
We then examine the computational complexity associated with forming and solving these equations in \cref{sec:flops}.

\subsection{Existence of solutions}\label{sec:existence}
We begin by demonstrating that the assumptions of \cref{thm:inputnormaloutputdiagonal} are sufficient to guarantee the existence of solutions to the polynomial input-normal/output-diagonal equations \cref{eq:input-normal-0,eq:output-diagonal-0}, or equivalently \cref{eq:input-normal-9,eq:output-diagonal-9}.
\begin{theorem}\label{thm:existence}
    Let the Jacobian linearization of the nonlinear system \cref{eq:FOM-NL} have distinct, nonzero Hankel singular values, as in \cref{thm:inputnormaloutputdiagonal}.
    Then a solution exists to the underdetermined linear algebraic equations \cref{eq:input-normal-9,eq:output-diagonal-9}.
\end{theorem}
\begin{proof}
    The coefficient matrix for the underdetermined system given by \cref{eq:input-normal-9,eq:output-diagonal-9}
    is the block matrix
    \begin{align*}
        \cA & \coloneqq \begin{bmatrix}
                            2\bN_k \\
                            2\hat\bN_k(\bI_{n^{k-1}} \otimes \bSigma^2)
                        \end{bmatrix}.
    \end{align*}
    A sufficient condition for the existence of a solution to an underdetermined linear system is that the coefficient matrix has full row-rank, in which case the right-hand-side vector of the linear system is trivially in the range $\cA$.

    Recall that the rows of $\bN_k$ and $\hat\bN_k$ are composed of 1s in the locations corresponding to equivalent monomials entries in $\kronF{\bz}{k}$ and 0s elsewhere.
    The first block of $\cA$, i.e. $2\bN_k$, has full row rank by construction then: each row corresponds to the equivalence class of a distinct monomial.
    A particular monomial can only belong to one equivalence class, so the rows are linearly independent and in fact orthogonal.
    The second block, $2\hat\bN_k(\bI_{n^{k-1}} \otimes \bSigma^2)$, similarly has full row rank by construction, since
    the rows of $\hat\bN_k$ are a subset of the rows of $\bN_k$.
    Since the Hankel singular values are assumed to be distinct and nonzero, the matrix $(\bI_{n^{k-1}} \otimes \bSigma^2)$ is a diagonal matrix with nonzero entries on the diagonal.
    Therefore, the rows of $\hat\bN_k(\bI_{n^{k-1}} \otimes \bSigma^2)$ are rows of $\bN_k$ scaled by values from $(\bI_{n^{k-1}} \otimes \bSigma^2)$.
    What remains then is to show that the rows of $\hat\bN_k(\bI_{n^{k-1}} \otimes \bSigma^2)$ are not linearly dependent on the rows of $\bN_k$.

    Let $\br_i$ be a row of $\bN_k$, and let $\hat\br_i$ be the row of $\hat\bN_k(\bI_{n^{k-1}} \otimes \bSigma^2)$ corresponding to the same monomial.
    Notice that $\hat\br_i = \br_i (\bI_{n^{k-1}} \otimes \bSigma^2)$, so each entry in $\hat\br_i$ is equal to the entry in $\br_i$ scaled by a particular squared Hankel singular value $\sigma_i^2$.
    Assume that $\hat\br_i$ is linearly dependent on $\br_i$; then each entry must be scaled by \emph{the same} $\sigma_i^2$, so we denote that squared Hankel singular value as $\sigma_\alpha^2$ and the index $\alpha \in [1,n]$.

    Looking more closely at the matrix $(\bI_{n^{k-1}} \otimes \bSigma^2)$,
    we see that
    its diagonal entries
    cyclically iterate through the squared Hankel singular values of indices $1$ through $n$:
    \begin{align*}
        (\bI_{n^{k-1}} \otimes \bSigma^2) & = \begin{bmatrix}
                                                  \sigma_1^2                                                \\
                                                   & \sigma_2^2                                             \\
                                                   &            & \ddots                                    \\
                                                   &            &        & \sigma_n^2                       \\
                                                   &            &        &            & \sigma_1^2          \\
                                                   &            &        &            &            & \ddots \\
                                              \end{bmatrix}.
    \end{align*}
    Similarly, the entries of the degree $k$ vector of monomials $\kronF{\bz}{k}$ contain at least one factor $z_j$ where $j$ cyclically iterates through the indices $1$ through $n$.
    This can be seen by denoting the $i$th entry of the vector $\kronF{\bz}{k-1}$ as $\zeta_{\scalebox{.5}{i}}$
    and writing the
    vector $\kronF{\bz}{k}$
    according to the definition of the Kronecker product
    as
    \begin{align*}
        \kronF{\bz}{k} = {\kronF{\bz}{k-1}} \otimes \bz = \begin{bmatrix}
                                                              \zeta_{\scalebox{.5}{1}} \, \bz            \\ 
                                                              \zeta_{\scalebox{.5}{2}} \, \bz            \\ 
                                                              \vdots                                     \\
                                                              \zeta_{\scalebox{.5}{\texttt{end}}} \, \bz \\ 
                                                          \end{bmatrix}, \quad \text{where }\bz = \begin{bmatrix}
                                                                                                      z_1 \\ z_2 \\\vdots \\ z_n
                                                                                                  \end{bmatrix}.
    \end{align*}
    So the index $j$ in the $m$th entry $\sigma_j^2$ along the diagonal $(\bI_{n^{k-1}} \otimes \bSigma^2)$ and the index $j$ of the \emph{last factor} $z_j$ of the $m$th monomial entry in $\kronF{\bz}{k}$ share a one-to-one correspondence.

    Now, recall that by its relation to $\bN_k$, the row $\br_i$ maps to an equivalence class of different permutations\footnote{For example, $x_1 x_2 x_3$, $x_1 x_3 x_2$, $x_2 x_1 x_3$, $x_2 x_3 x_1$, $x_3 x_1 x_2$, $x_3 x_2 x_1$.}
    of the same multivariate monomial of degree $k$ contained in the vector of monomials $\kronF{\bz}{k}$.
    Assume then that for a particular row $\br_i$,
    the corresponding row $\hat\br_i = \br_i (\bI_{n^{k-1}} \otimes \bSigma^2)$
    picks out all copies of the same squared Hankel singular value $\sigma_\alpha^2$, as required for linear dependence.
    Then due to the one-to-one correspondence between the squared Hankel singular values $\sigma_j^2$ in $(\bI_{n^{k-1}} \otimes \bSigma^2)$ and the last factor $z_j$ in the entries of $\kronF{\bz}{k}$, that particular row
    $\br_i$
    would also map to entries of $\kronF{\bz}{k}$ that all contain $z_\alpha$ as the last factor in the monomial.
    In other words, this implies that all permutations of the degree $k$ monomial in this equivalence class contain the same factor $z_\alpha$ in the last position.

    We have therefore reached a contradiction: the rows $\hat\br_i$ corresponding to rows of $\hat\bN_k$ map to all of the off-diagonal monomials.
    By definition, the off-diagonal monomials contain at least two distinct factors $z_\alpha$ and $z_\beta$ where $\alpha \neq \beta$.
    By permuting the factors $z_\beta$ and $z_\alpha$, a monomial with $z_\beta$ in the final position can be formed which still belongs to the same permutation equivalence class mapped to by the row $\hat\br_i$.
    Hence the index corresponding to this monomial permutation was not included in the row of $\hat\bN_k$, implying that $\hat\bN_k$ does not match its definition.

    We can therefore conclude that at least one entry in the row $\hat\br_i$ is scaled by a squared Hankel singular value $\sigma_\beta^2$ that is \emph{distinct} from $\sigma_\alpha^2$, implying that $\hat\br_i$ is \underline{not} linearly dependent on $\br_i$ and that the coefficient matrix $\cA$ has full row rank.
    Thus the linear system given by \cref{eq:input-normal-9,eq:output-diagonal-9} has a solution.
\end{proof}

\subsection{FLOP count}\label{sec:flops}
In this section, we investigate the cost of computing the polynomial input-normal/output-diagonal transformation of \cref{prop:sparse computation}.
\begin{proposition}
    Let the coefficients of the controllability and observability energy functions be given to degree $d$.
    The computational complexity of forming equations \cref{eq:input-normal-9,eq:output-diagonal-9} to compute a degree $k=d-1$ transformation to balance the energy functions to degree~$d$ is $O(dn^{d+1})$.
\end{proposition}
\begin{proof}
    Assuming we wish to balance energy functions of degree $d$, we evaluate the complexity of computing a degree $k=d-1$ transformation.
    First, we must consider the cost of computing the square-root factors $\bV_2 = \bR \bR^\top$ and $\bW_2 = \bL \bL^\top$ and computing the SVD of $\bL^\top \bR^{-\top} = \cU \bSigma \cV^\top$; these operations are all $O(n^3)$.
    Additionally, the linear transformation coefficient $\bT_1$ is formed by multiplication of two $n \times n$ matrices, again with a cost of $O(n^3)$.

    Next, consider the cost of transforming the energy function coefficients by $\bT_1$, given by
    $\tilde\bv_i=(\bT_1 \otimes \bT_1 \otimes \dots \otimes\bT_1)^\top \bv_i$ and $\tilde\bw_i=(\bT_1 \otimes \bT_1 \otimes \dots \otimes\bT_1)^\top \bw_i$.
    The most expensive coefficient to evaluate is the $d$th coefficient.
    These products can be evaluated efficiently using recursion and the Kronecker-vec relation
    $(\bB^\top \otimes \bA) \Vec[\bD]= \Vec[\bA \bD \bB]$
    \cite{Golub2013}.
    At the first level of recursion, we apply the identity once and compute
        {\nomathindent\begin{align*}
                \cT_{d,d}^\top \bv_{d}  = {\kronF{\bT_1}{d}}^\top \bv_{d}
                 & =
                \Vec\left[{\kronF{\bT_1}{d-1}}^\top {\rm{reshape}}(\bv_{d},n^{d-1},n) \bT_1\right].
            \end{align*}}
    The matrix-matrix product $\bG \coloneqq {\rm{reshape}}(\bv_{d},n^{d-1},n) \bT_1$ costs $O(n^{d+1})$ and results in an $n^{d-1} \times n$ matrix $\bG$.
    We then apply a second level of recursion to compute
    $
        {\kronF{\bT_1}{d-1}}^\top \bG
    $
    using the same Kronecker-vec identity on each of the $n$ columns of $\bG$.
    We reshape each column $\bg_i \in \real^{n^{d-1}}$ as an $n^{d-2} \times n$ matrix to compute
    \begin{align*}
         & {\kronF{\bT_1}{d-1}}^\top \bg_i
        = \Vec\left[{\kronF{\bT_1}{d-2}}^\top {\rm{reshape}}(\bg_i,n^{d-2},n) \bT_1 \right],
    \end{align*}
    which involves matrix-matrix multiplication costing $O(n^{d})$ and results in an $n^{d-2} \times n$ matrix $\bH \coloneqq {\rm{reshape}}(\bg_i,n^{d-2},n) \bT_1$.
    Performing this $n$ times for each column of $\bG$, the second level of recursion costs $O(n^{d+1})$.
    We then use a third level of recursion to compute
    \begin{align*}
         & {\kronF{\bT_1}{d-2}}^\top \bh_i
        = \Vec\left[{\kronF{\bT_1}{d-3}}^\top {\rm{reshape}}(\bh_i,n^{d-3},n) \bT_1 \right]
    \end{align*}
    column by column.
    Each of the $n$ columns involves matrix-matrix multiplication costing $O(n^{d-1})$.
    But each $\bH$ corresponds to one of the $n$ columns of $\bG$, so in total $n^2$ matrix-matrix multiplications costing $O(n^{d-1})$ are performed.
    Hence the third level of recursion costs $O(n^{d+1})$.

    One can see that the $m$th level of recursion involves $n^{m-1}$ matrix-matrix multiplications costing $O(n^{d-m+2})$, so the overall cost of each level of recursion is always $O(n^{d+1})$.
    Since $d$ recursion steps are required, the total cost of transforming the coefficients is $O(dn^{d+1})$.

    We continue considering the cost of forming the remaining terms in the input-normal/output-diagonal equations \cref{eq:input-normal-9,eq:output-diagonal-9}.
    The matrix multiplication $\hat\bT_{j}^\top \bSigma^2 \hat\bT_{i}$ costs $O(n^{d+1})$ since $i+j=d$, and we compute $d$ instances of this quantity in the sum for a total cost of $O(dn^{d+1})$.
    Note that $\sum_{\substack{i,j\geq 1 \\ i+j=d}}\hat\bT_{j}^\top \bSigma^2 \hat\bT_{i} = \hat\cT_{2,d}^\top \tilde\bw_2$, which can be shown using the same Kronecker-vec relation $\Vec[\bA \bD \bB] = (\bB^\top \otimes \bA) \Vec[\bD]$ as before.
    Using similar arguments to those used to far, one can show that in fact all of the terms in the sum $\sum_{i=3}^d  \hat\cT_{i,d}^\top \tilde\bw_i$ can be formed using the Kronecker-vec relation for a total cost of $O(dn^{d+1})$.
    Hence the cost of forming the input-normal/output-diagonal equations is $O(dn^{d+1})$.
\end{proof}
The most expensive coefficient to compute is the final coefficient $\bT_{d-1}$, which contains $n^d$ unknowns.
A naive solution to the dense linear system given by \cref{eq:input-normal-0,eq:output-diagonal-0} requires $O(n^{3d})$ operations.
However, by transforming first with $\bT_1$, the equivalent linear system \cref{eq:input-normal-9,eq:output-diagonal-9} is sparse.
Furthermore, since $\bT_{d-1}$ is a transformation coefficient that acts on $\kronF{\bz}{d-1}$, the rows of $\bT_{d-1}$ are only unique up to an equivalence class corresponding to symmetrization.
Each of the $n$ rows of $\bT_{d-1}$ only needs to contain ${n+d-2}\choose{d-1}$ unique entries, for a total of $n {{n+d-2}\choose{d-1}}$ unknowns in the coefficient $\bT_{d-1}$.
In practice, it was found that the SuiteSparseQR solver \cite{Davis2011} behind \textsc{Matlab}'s built-in ``backslash'' command is sufficient to avoid exceeding $O(dn^{d+1})$ computational complexity.
An alternative is to compute the minimum norm solution, e.g using \textsc{Matlab}'s \texttt{lsqminnorm()} solver;
we did not notice any benefit to using a minimum norm solution instead of the SuiteSparseQR solution on the examples we considered, so we favored the sparse solutions of SuiteSparseQR.

It is difficult to estimate the flop count of sparse solvers, so in \cref{sec:results} we demonstrate empirically that with this approach, forming and solving the linear systems \cref{eq:input-normal-9,eq:output-diagonal-9} scales similarly to $O(dn^{d+1})$.
Notably, computing the energy function coefficients in \cref{eq:vi-coeffs} using Algorithm 1 from \cite{Corbin2023} also has a cost of $O(dn^{d+1})$, so
the cost of computing a polynomial input-normal/output-diagonal transformation using the proposed approach is commensurate with
the cost of computing the energy functions themselves.

\section{Numerical results}\label{sec:results}
In this section, we demonstrate our proposed approach first with a two-dimensional academic example, and then with a scalable system of $N$ nonlinear mass-spring-dampers.
We demonstrate that, when the assumptions underlying the theory are satisfied, our approach produces a polynomial approximation of an input-normal/output-diagonal transformation.
Furthermore, we show that our approach scales similarly to the cost of computing the energy function coefficients.
The results are obtained on a
Linux workstation with an Intel Xeon W-3175X CPU, 256 GB RAM, and \textsc{Matlab} 2021a.
All of the examples, along with the efficient implementation of \cref{prop:sparse computation}, are provided in the \texttt{cnick1/NLbalancing} repository \cite{NLBalancing2023}.

\subsection{Illustrative two-dimensional example}
Consider the following 2D model introduced first in \cite{Fujimoto2001a} and then studied again in \cite{Fujimoto2005,Fujimoto2010}:
\begin{align*}
    \bf(\bx) & = \begin{pmatrix}
                     -9x_1 + 6x_1^2 x_2 + 6x_2^3 - x_1^5 - 2x_1^3 x_2^2 - x_1 x_2^4 \\
                     -9x_2 - 6x_1^3 - 6x_1x_2^2 - x_1^4 x_2 - 2x_1^2 x_2^3 - x_2^5
                 \end{pmatrix}    ,                                                                    \\
    \bg(\bx) & = \left(\begin{matrix}
                           \frac{3\sqrt{2}(9-6x_1x_2+x_1^4-x_2^4)}{9+x_1^4+2x_1^2x_2^2+x_2^4} \\
                           \frac{\sqrt{2}(27x_1^2+9x_2^2+6x_1^3x_2+6x_1x_2^3+(x_1^2+x_2^2)^3}{9+x_1^4+2x_1^2x_2^2+x_2^4}
                       \end{matrix}\right.                                \\
             & \hspace{1.5cm} \left.\begin{matrix}
                                        \frac{\sqrt{2}(-9x_1^2 - 27 x_2^2 + 6 x_1^3 x_2 + 6 x_1 x_2^3 - (x_1^2 + x_2^2)^3)}{9+x_1^4+2x_1^2x_2^2+x_2^4} \\
                                        \frac{3\sqrt{2}(9 + 6 x_1 x_2  - x_1^4 + x_2^4)}{9+x_1^4+2x_1^2x_2^2+x_2^4}
                                    \end{matrix}\right), \\
    \bh(\bx) & = \begin{pmatrix}
                     \frac{2\sqrt{2}(3x_1 + {x_1^2 x_2} + x_2^3)(3 - x_1^4 - 2x_1^2 x_2^2 - x_2^4)}{1 + x_1^4 + 2 x_1^2 x_2^2 + x_2^4} \\
                     \frac{\sqrt{2}(3x_2 - x_1^3 - x_1 x_2^2)(3 - x_1^4 - 2 x_1^2 x_2^2 - x_2^4)}{1 + x_1^4 + 2 x_1^2 x_2^2 + x_2^4}
                 \end{pmatrix}.
\end{align*}
In \cite{Fujimoto2001a}, symbolic computations are used to obtain the following analytical expressions for the controllability and observability energy functions:
{\nomathindent\begin{align*}
    \cE_c(\bx) & =  \frac{1}{2} (x_1^2 + x_2^2) ,                                                                                                                                              \\
    \cE_o(\bx) & = \text{\footnotesize $\frac{36 x_1^2 + 9 x_2^2 + 18 x_1^3 x_2 + 18 x_1 x_2^3 + x_1^6 + 6 x_1^4 x_2^2 + 9 x_1^2 x_2^4 + 4 x_2^6}{2(1 + x_1^4 + 2   x_1^2   x_2^2 + x_2^4)}$}.
\end{align*}}

Using Algorithm 1 from \cite{Corbin2023}, we compute the Taylor expansion of the energy functions by solving the HJB PDEs.
The following results match the Taylor expansions of the analytical energy functions to degree 6 within $O(10^{-13})$:
{\nomathindent\begin{align*}
    \cE_c(\bx) & = \frac{1}{2} (x_1^2 + x_2^2),       \\
    \cE_o(\bx) & = \text{\footnotesize $\frac{1}{2} (
            36 x_1^2
            + 9 x_2^2
            + 18 x_1^3 x_2
            + 18 x_1 x_2^3
            - 35 x_1^6
            - 75 x_1^4 x_2^2
            - 45 x_1^2 x_2^4) .$}
\end{align*}}

Next, we compute the input-normal/output-diagonal transformation $\bx = \bPhi(\bz)$ using \cref{prop:sparse computation}; since $n=2$,
the input-normal/output-diagonal transformation equations are not underdetermined, so the transformation is unique.
Indeed, our transformation is consistent with the results presented in \cite{Fujimoto2001a,Fujimoto2005,Fujimoto2010}:
{\nomathindent\begin{align*}
    \bPhi(\bz) = \begin{bmatrix}
                     z_1 - \frac{1}{3} z_2^3 - \frac{1}{3} z_1^2 z_2 - \frac{1}{18}  z_1^5 + \frac{11}{9} z_1^3 z_2^2 + \frac{5}{6} z_1 z_2^4 \\
                     z_2 + \frac{1}{3} z_1^3 + \frac{1}{3} z_1 z_2^2 - \frac{1}{18}  z_2^5 - \frac{25}{18} z_1^4 z_2 -  z_1^2 z_2^3
                 \end{bmatrix}.
\end{align*}}
The singular value functions are simply given by the observability energy in these transformed coordinates; for this problem, they are $\sigma_1^2(z_1) = 36 - 32 z_1^4$ and $\sigma_2^2(z_2) =  9 - 8 z_2^4$.
This example demonstrates that our proposed approach correctly computes a polynomial input-normal/output-diagonal transformation.
Next we present a scalable example with arbitrary dimension $n$ to demonstrate how the proposed method can be applied to larger models.

\subsection{$N$ coupled Duffing oscillators}
Consider a system of $N$ coupled nonlinear mass-spring-dampers as shown in \cref{fig:example17-msd}.
Instead of linear springs which provide a force proportional to their deformation $F_\text{lin} = -k\Delta x$, we use nonlinear springs with a restoring force $F_\text{nl} = -k\sin \Delta x$, inspired by the pendulum.
Using a cubic approximation $\sin(x) \approx x - \frac{x^3}{6}$ leads to a system of coupled Duffing oscillators.
\begin{figure}[htb]
    \centering
    \includegraphics[width=\columnwidth]{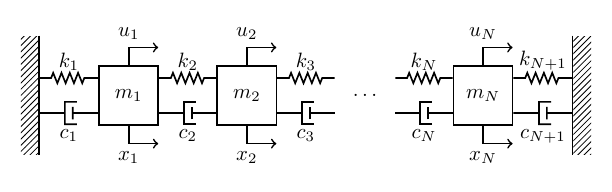}
    \caption{System of $N$ coupled nonlinear mass-spring-dampers.}
    \label{fig:example17-msd}
\end{figure}

For simplicity, we set $m_i = k_i = c_i = 1$; the inputs are forces applied to each mass, and the outputs are the position measurements of each mass.
The system can then be represented by a state-space model of dimension $n=2N$ with cubic polynomial form:
\begin{equation}
    \dot{\bx}  = \bA \bx + \F_3 \kronF{\bx}{3} + \bB \bu, \qquad
    \by        = \bC \bx.
\end{equation}
Since the dynamics are strictly odd, the energy functions will be strictly even, and the input-normal/output-diagonal transformation will be strictly odd as well.

\subsubsection{Diagonalization Performance}
To concisely demonstrate that the transformation puts the energy functions in input-normal/output-diagonal form, \cref{tab:combined-norms} reports various normed quantities.
For the transformed controllability energy, the quadratic coefficient should be equivalent to an identity matrix, and the quartic coefficient should be zero.
For the transformed observability energy, each of the coefficients should be diagonal; hence the off-diagonal entries should be zero.
\begin{table*}[htb]
    \centering
    \caption{Error metrics depicting how close the transformed energy function coefficients are to input-normal/output-diagonal form. The controllability coefficients should be $\tilde\bV_2 = \bI_n$ and $\tilde\bv_4 = \bzero$; the off-diagonal entries in the observability coefficients should be zero.
    }
    \begin{tabular}{ccccc}
        \toprule
        n  & $\left\lVert \tilde\bv_2 - \text{vec}(\bI_n)\right\rVert_2$ & $\left\lVert \tilde\bv_4\right\rVert_2$ & $\left\lVert \text{offdiag}(\tilde\bw_2)\right\rVert_2$ & $\left\lVert \text{offdiag}(\tilde\bw_4)\right\rVert_2$ \\
        \midrule
        6  & \texttt{1.5e-15}                                            & \textbf{\texttt{1.6e-04}}               & \texttt{3.5e-16}                                        & \textbf{\texttt{1.3e-03}}                               \\
        8  & \texttt{3.1e-15}                                            & \texttt{4.8e-16}                        & \texttt{9.5e-16}                                        & \texttt{1.4e-16}                                        \\
        16 & \texttt{4.6e-15}                                            & \texttt{1.0e-14}                        & \texttt{5.9e-14}                                        & \texttt{1.1e-14}                                        \\
        32 & \texttt{8.3e-15}                                            & \texttt{9.7e-14}                        & \texttt{2.5e-12}                                        & \texttt{8.1e-13}                                        \\
        50 & \texttt{1.4e-14}                                            & \texttt{4.3e-13}                        & \texttt{3.8e-12}                                        & \texttt{3.5e-11}                                        \\
        64 & \texttt{1.5e-14}                                            & \textbf{\texttt{1.1e-06}}               & \texttt{6.6e-12}                                        & \textbf{\texttt{3.0e-05}}                               \\
        \bottomrule
    \end{tabular}
    \label{tab:combined-norms}
\end{table*}
As seen in \cref{tab:combined-norms}, the proposed approach works for model dimensions up to $n=50$ for this example.
The cases $n=6$ and $n=64$ illustrate two scenarios in which the assumptions of \cref{thm:inputnormaloutputdiagonal} break down, which ultimately violates the assumptions required to prove existence of solutions to the transformation in \cref{thm:existence}.
Consider first the case of $n=6$;
the Hankel singular values for the linearized system are 1.2071, 0.5000, 0.3536, 0.3536, 0.2500, 0.2071.
\cref{thm:inputnormaloutputdiagonal} requires distinct Hankel singular values, but the middle two Hankel singular values are identical.
As a result, we see in \cref{tab:combined-norms} that the transformed observability energy fails to be diagonalized.

For state dimensions larger than $n=50$ for this example, ill-conditioning due to small Hankel singular values becomes problematic.
Since many Hankel singular values are numerically close to zero, it becomes more likely that two of these very small Hankel singular values are close to each other,
and the distinct Hankel singular value assumption is again violated.
As a result, the computed transformations fail to impose the input-normal/output-diagonal structure, as shown in \cref{tab:combined-norms} for the case of $n=64$.

\cref{fig:n6-2D-3D-projections,fig:n8-2D-3D-projections,fig:n50-2D-3D-projections} depict visualizations of the quartic observability energy coefficient $\tilde\bw_4$ for $n=6,8$ and $50$.
Since $\tilde\bw_4$ is a 4D tensor, it is impossible to visualize it directly; however, since the tensor is symmetric, diagonal structure in 2D or 3D implies diagonal structure in the original dimension and vice versa.
Therefore, we apply the 2-norm along fibers of the tensor to collapse down to dimensions that can be visualized.
In these figures, shade and opacity depict the relative magnitude of each entry in the tensor.

Consider first the $n=6$ case shown in \cref{fig:n6-2D-3D-projections}, which fails to achieve input-normal/output-diagonal structure due to repeated Hankel singular values.
One can see in the 3D and 2D projections of the transformed coefficient $\tilde\bw_4$ that some of the off-diagonal components remain nonzero.
Contrast this with the $n=8$ case shown in \cref{fig:n8-2D-3D-projections}, which exhibits no visible off-diagonal entries in $\tilde\bw_4$.
\begin{figure}[htb]
    \centering
    \begin{subfigure}[h]{0.325\columnwidth}
        \centering
        \includegraphics[width=\textwidth,page=6]{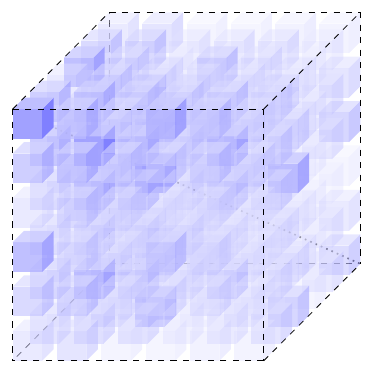}
        \caption{3D projection of $\bw_4$.}
    \end{subfigure}
    \hfill
    \begin{subfigure}[h]{0.325\columnwidth}
        \centering
        \includegraphics[width=\textwidth,page=8]{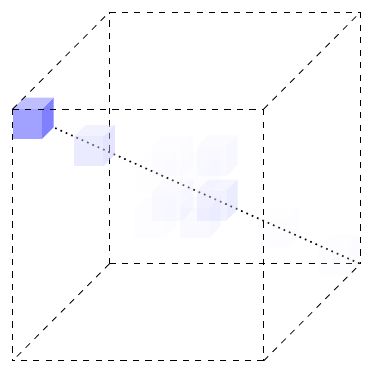}
        \caption{3D projection of $\tilde\bw_4$.}
    \end{subfigure}
    \hfill
    \begin{subfigure}[h]{0.25\columnwidth}
        \centering
        \vspace{.6cm}
        \includegraphics[width=\textwidth,page=7]{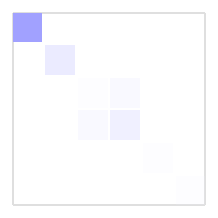}
        \caption{2D projection of $\tilde\bw_4$.}
    \end{subfigure}
    \caption{(Case $n=6$) 3D and 2D projection of $\bw_4$ (original coordinates) and $\tilde\bw_4$ (transformed coordinates), which can be regarded as a symmetric 4D tensor. As a result of repeated Hankel singular values, $\tilde\bw_4$ still has off-diagonal entries.}
    \label{fig:n6-2D-3D-projections}
\end{figure}

\begin{figure}[htb]
    \centering
    \begin{subfigure}[h]{0.325\columnwidth}
        \centering
        \includegraphics[width=\textwidth,page=10]{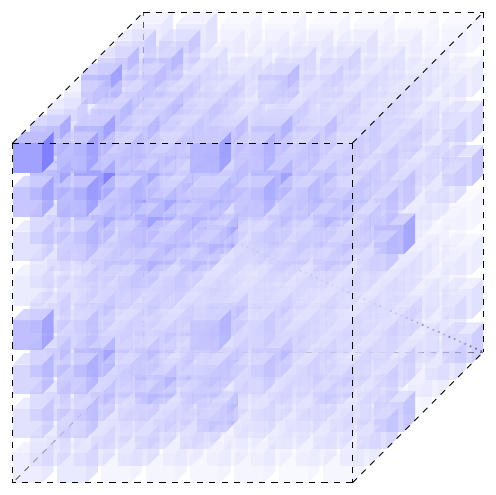}
        \caption{3D projection of $\bw_4$.}
    \end{subfigure}
    \hfill
    \begin{subfigure}[h]{0.325\columnwidth}
        \centering
        \includegraphics[width=\textwidth,page=12]{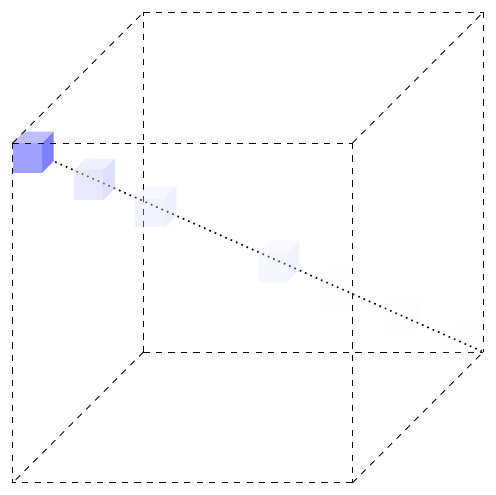}
        \caption{3D projection of $\tilde\bw_4$.}
    \end{subfigure}
    \hfill
    \begin{subfigure}[h]{0.25\columnwidth}
        \centering
        \vspace{.6cm}
        \includegraphics[width=\textwidth,page=11]{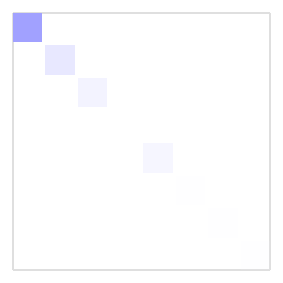}
        \caption{2D projection of $\tilde\bw_4$.}
    \end{subfigure}
    \caption{ (Case $n=8$) 3D and 2D projection of $\bw_4$ (original coordinates) and $\tilde\bw_4$ (transformed coordinates), which can be regarded as a symmetric 4D tensor.}
    \label{fig:n8-2D-3D-projections}
\end{figure}

Similarly, the $n=50$ case succeeds in diagonalization, as shown in \cref{fig:n50-2D-3D-projections}.
Here we omit the 3D projection since it is too cluttered.
The $n=64$ case, similar to the $n=6$ case, retains off-diagonal components due to the failure to satisfy the distinct Hankel singular values assumption.
\begin{figure}[htb]
    \centering
    \begin{subfigure}[h]{0.475\columnwidth}
        \centering
        \includegraphics[width=\textwidth,page=2]{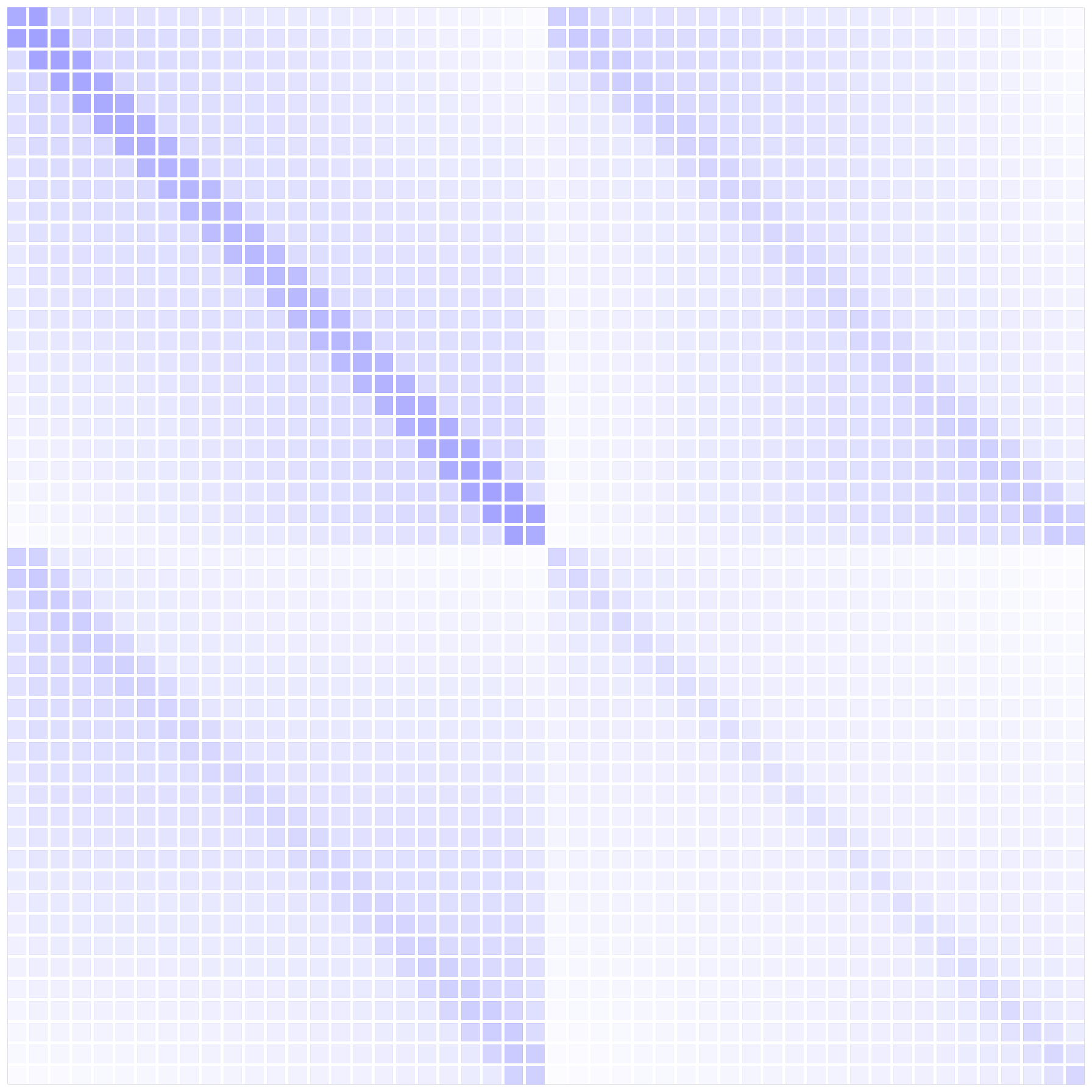}
        \caption{2D projection of $\bw_4$.}
    \end{subfigure}
    \hfill
    \begin{subfigure}[h]{0.475\columnwidth}
        \centering
        \includegraphics[width=\textwidth,page=4]{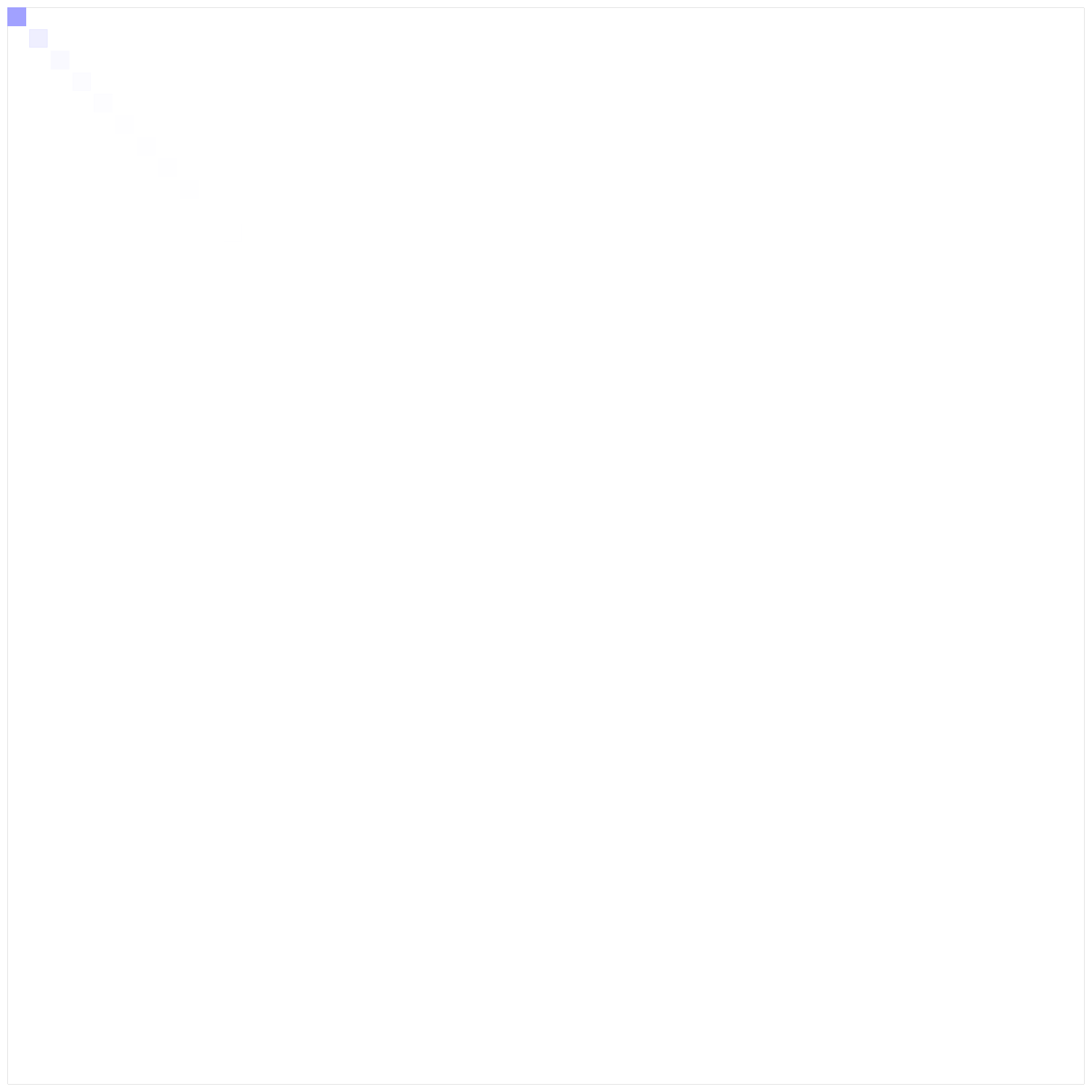}
        \caption{2D projection of $\tilde\bw_4$.}
    \end{subfigure}
    \caption{2D projection of $\bw_4$ (original coordinates) and $\tilde\bw_4$ (transformed coordinates), which can be regarded as a symmetric 4D tensor ($n=50$).}
    \label{fig:n50-2D-3D-projections}
\end{figure}
These results illustrate that, when the necessary assumptions posed by the nonlinear balancing theory are satisfied, our proposed algorithms are able to compute polynomial transformations that induce input-normal/output-diagonal structure.

\subsubsection{Scaling}
In \cref{sec:flops}, we claim that using our proposed approach and exploiting sparsity,
a degree $d-1$ transformation to put the system in input-normal/output-diagonal form to degree $d$
can be computed with a computational complexity of $O(dn^{d+1})$.
Critically, this is on par with the computational complexity of computing a degree $d$ approximation to the energy function, which also scales as $O(dn^{d+1})$.
In this section, we demonstrate that our implementation matches this performance.

\cref{fig:example17} depicts the CPU time
for computing energy function approximations of degree~3, 4, 5, and 6 and the corresponding degree~2, 3, 4, and 5 input-normal/output-diagonal transformations.
Compare the computational time associated with forming \& solving for the transformation with the computational time associated with solving for the energy function coefficients;
we estimate the theoretical scaling of both of these as $O(dn^{d+1})$.
In each of the cases in \cref{fig:example17}, both lines are almost parallel, indicating that they scale at a similar rate.
Thus, our proposed method can compute nonlinear balancing \emph{transformations} with scalability that is comparable to the state-of-the-art methods for computing nonlinear balancing \emph{energy functions}.
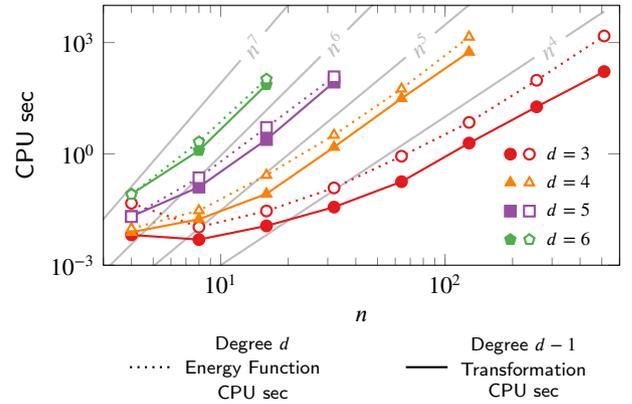
\begin{figure}[htb]
    \centering
    \begin{tikzpicture}
        \begin{loglogaxis}[
                xlabel={$n$},
                ylabel={CPU sec},
                width=\columnwidth,
                height=5cm,
                xmin=3, xmax=600, ymin=1e-3, ymax=1e4,axis on top,
                ytickten={-3,0,3},
                legend pos=south east,
                legend style={draw=none,fill=none,font=\sffamily\scriptsize,cells={align=center}},
            ]
            \addplot[domain=2:512, samples=10, color=lightgray, thick, forget plot] {10^(-7) * x^4} node [pos=.9, fill=white, sloped] {$n^{4}$};
            \addplot[domain=2:512, samples=10, color=lightgray, thick, forget plot] {3*10^(-7) * x^5} node [pos=.66, fill=white, sloped] {$n^{5}$};
            \addplot[domain=2:512, samples=10, color=lightgray, thick, forget plot] {9*10^(-7) * x^6} node [pos=.5, fill=white, sloped] {$n^{6}$};
            \addplot[domain=2:512, samples=10, color=lightgray, thick, forget plot] {8*10^(-6) * x^7} node [pos=.35, fill=white, sloped] {$n^{7}$};

            \addplot[thick,color=Set1-A,mark=*,mark size=2] table[x index=0, y index=2,col sep=&,forget plot] {example17_averagedTimings_d3.dat};
            \addplot[thick,color=Set1-A,mark=*,mark size=2,dotted,mark options={solid, fill=white}] table[x index=0, y index=1,col sep=&] {example17_averagedTimings_d3.dat};

            \addplot[thick,color=Set1-E,mark=triangle*,mark size=2] table[x index=0, y index=2,col sep=&,forget plot] {example17_averagedTimings_d4.dat};
            \addplot[thick,color=Set1-E,mark=triangle*,mark size=2,dotted,mark options={solid, fill=white}] table[x index=0, y index=1,col sep=&] {example17_averagedTimings_d4.dat};

            \addplot[thick,color=Set1-D,mark=square*,mark size=2] table[x index=0, y index=2,col sep=&,forget plot] {example17_averagedTimings_d5.dat};
            \addplot[thick,color=Set1-D,mark=square*,mark size=2,dotted,mark options={solid, fill=white}] table[x index=0, y index=1,col sep=&] {example17_averagedTimings_d5.dat};

            \addplot[thick,color=Set1-C,mark=pentagon*,mark size=2] table[x index=0, y index=2,col sep=&,forget plot] {example17_averagedTimings_d6.dat};
            \addplot[thick,color=Set1-C,mark=pentagon*,mark size=2,dotted,mark options={solid, fill=white}] table[x index=0, y index=1,col sep=&] {example17_averagedTimings_d6.dat};
        \end{loglogaxis}

        \node at (3.4, -0.6) [anchor=north] {
            \begin{tikzpicture}
                \matrix[row sep=3mm,column sep=0.5mm] {
                    \draw[thick, black, dotted] (0,0) -- (0.55,0); & \node[anchor=west] {\sffamily\scriptsize \shortstack{Degree $d$ \\ Energy Function \\ CPU sec}};
                                                                   &                                                                    &  &  &  &  &  &  &  &  &  &  &  &  &  &  &  &  &  &  &
                    \draw[thick, black] (0,0) -- (0.55,0);         & \node[anchor=west] {\sffamily\scriptsize \shortstack{Degree $d-1$ \\ Transformation \\ CPU sec}};                                                         \\
                };
            \end{tikzpicture}
        };

        \node at (5.9, 1.9) [anchor=north] {
            \begin{tikzpicture}
                \matrix[row sep=-0.15mm,column sep=0.5mm] {
                    \draw plot[only marks,mark=*,mark size=2,mark options={thick,fill=Set1-A,draw=Set1-A}] coordinates{(0.75,-0.02)};
                    \draw plot[only marks,mark=*,mark size=2,mark options={thick,fill=white,draw=Set1-A}] coordinates{(1,-0.02)};
                     & \node[anchor=west] {\sffamily\scriptsize $d=3$}; \\
                    \draw plot[only marks,mark=triangle*,mark size=2,mark options={thick,fill=Set1-E,draw=Set1-E}] coordinates{(0.75,-0.02)};
                    \draw plot[only marks,mark=triangle*,mark size=2,mark options={thick,fill=white,draw=Set1-E}] coordinates{(1,-0.02)};
                     & \node[anchor=west] {\sffamily\scriptsize $d=4$}; \\
                    \draw plot[only marks,mark=square*,mark size=2,mark options={thick,fill=Set1-D,draw=Set1-D}] coordinates{(0.75,-0.02)};
                    \draw plot[only marks,mark=square*,mark size=2,mark options={thick,fill=white,draw=Set1-D}] coordinates{(1,-0.02)};
                     & \node[anchor=west] {\sffamily\scriptsize $d=5$}; \\
                    \draw plot[only marks,mark=pentagon*,mark size=2,mark options={thick,fill=Set1-C,draw=Set1-C}] coordinates{(0.75,-0.02)};
                    \draw plot[only marks,mark=pentagon*,mark size=2,mark options={thick,fill=white,draw=Set1-C}] coordinates{(1,-0.02)};
                     & \node[anchor=west] {\sffamily\scriptsize $d=6$}; \\
                };
            \end{tikzpicture}
        };
    \end{tikzpicture}
    \caption{Scaling of CPU time for both the energy function computation and the transformation computation as $n$ increases for $d=3,4,5,6$ approximations.
        The transformation CPU time scales similarly to the energy function CPU time, which is $O(dn^{d+1})$.}
    \label{fig:example17}
\end{figure}

\section{Conclusion}\label{sec:conclusion}
We have developed and analyzed a tensor-based approach to computing polynomial input-normal/output-diagonal transformations required for nonlinear balanced truncation model reduction.
The code is freely available online in the \texttt{cnick1/NLbalancing} repository \cite{NLBalancing2023}.
Our method involves solving a set of underdetermined algebraic equations whose solutions are the coefficients associated with the transformation.
We derive the explicit form for the equations, prove the existence of solutions to those equations, and devise an efficient two-step computational scheme for computing the transformation through \cref{prop:sparse computation}.
In addition to proving the existence of solutions to the equations for the transformation equations and discussing the family of solutions that exist, we analyze the computational complexity of the proposed approach and demonstrate its scalability on examples.
As evident from the numerical results, the computational time associated with solving for the polynomial transformation via our approach scales at a similar rate as solving for the energy functions using the approach of \citet{Corbin2023}, representing a significant step towards a viable implementation of nonlinear balanced truncation model reduction for control-affine nonlinear systems.

The primary direction of future work involves constructing balanced nonlinear ROMs in the transformed coordinates and comparing with other model reduction techniques.
In this work, we have presented a method for computing the full transformation to put the system in input-normal/output-diagonal form; similar to the linear case, we found that computing the full transformation is often numerically ill-conditioned for large models due to the presence of small Hankel singular values.
Hence, as in the linear case, we plan to develop a reduced transformation approach (sometimes referred to balance-\emph{and}-reduce as opposed to balance-\emph{then}-reduce) for the nonlinear case.

\section*{Acknowledgments}
This work was supported by the National Science Foundation under
Grant CMMI-2130727.
The research of Arijit Sarkar is part of the project Digital Twin P18-03 project 1 of the research program Perspectief which is (mainly) financed by the Dutch Research Council (NWO).
\printcredits

\bibliographystyle{cas-model2-names}
\bibliography{references}

\end{document}